\documentclass[5p,preprint]{elsarticle}
\usepackage[utf8]{inputenc}
\usepackage[T1]{fontenc}
\usepackage{amsmath}
\usepackage{amssymb}
\usepackage{verbatim}
\usepackage{pdfpages}
\usepackage[colorlinks=true]{hyperref}
\usepackage{xcolor}
\usepackage{pst-all}
\usepackage{pstricks-add}
\usepackage{tabto}
\usepackage{leftidx}
\usepackage{xstring}
\usepackage{todonotes}
\usepackage[shortlabels]{enumitem}
\usepackage{dsfont}
\usepackage{marvosym}
\usepackage{stmaryrd}
\usepackage{multicol}
\usepackage{graphicx}
\usepackage{chngcntr}
\usepackage{epsfig}
\usepackage{tabularx}
\usepackage{mathtools}
\usepackage{mathrsfs}
\usepackage{float}
\usepackage{subcaption}
\usepackage{tikz}
\usetikzlibrary{backgrounds}
\usetikzlibrary{intersections}
\usepackage{tkz-euclide}

    \newif\ifCustomTheorems
    \usepackage{amsthm}
\usepackage{cleveref}

\theoremstyle{definition}

\makeatletter
\newcommand\RedeclareMathOperator{%
  \@ifstar{\def\rmo@s{m}\rmo@redeclare}{\def\rmo@s{o}\rmo@redeclare}%
}
\newcommand\rmo@redeclare[2]{%
  \begingroup \escapechar\m@ne\xdef\@gtempa{{\string#1}}\endgroup
  \expandafter\@ifundefined\@gtempa
     {\@latex@error{\noexpand#1undefined}\@ehc}%
     \relax
  \expandafter\rmo@declmathop\rmo@s{#1}{#2}}
\newcommand\rmo@declmathop[3]{%
  \DeclareRobustCommand{#2}{\qopname\newmcodes@#1{#3}}%
}
\@onlypreamble\RedeclareMathOperator
\makeatother

\newcommand{\N}{\mathds{N}}
\newcommand{\R}{\mathds{R}}
\newcommand{\Rp}{\R_{\geq0}}

\newcommand{\Impl}{\Longrightarrow}

\newcommand{\fa}{\ \forall \, }
\newcommand{\ex}{\ \exists \, }

\newcommand{\rbl}{\left (}
\newcommand{\rbr}{\right )}

\newcommand{\al}{\left \langle}
\newcommand{\ar}{\right \rangle}
\newcommand{\bl}{\left |}
\newcommand{\br}{\right |}
\newcommand{\nl}{\left\|}
\newcommand{\nr}{\right\|}
\newcommand{\cbl}{\left\lbrace }
\newcommand{\cbr}{\right\rbrace }
\newcommand{\Abs}[1]{\bl #1 \br}
\newcommand{\Norm}[2][ ]{\nl #2 \nr_{#1}}
\newcommand{\SNorm}[1]{\Norm[\infty]{#1}}
\newcommand{\LNorm}[2][2]{\Norm[L^{#1}]{#2}}

\newcommand{\setdef}[2]{\cbl\ #1\ \left|\ \vphantom{#1} #2\ \right.\cbr}

\newcommand{\GL}{\text{GL}}

\newcommand{\cD}{\mathcal{D}}

\newcommand{\cU}{\mathcal{U}}

\newcommand{\cF}{\mathcal{F}}
\newcommand{\cG}{\mathcal{G}}
\newcommand{\cN}{\mathcal{N}}
\newcommand{\cY}{\mathcal{Y}}

\newcommand{\cT}{\mathcal{T}}

\DeclareMathOperator*{\rf}{ref}
\DeclareMathOperator*{\esssup}{ess\,sup}

\DeclareMathOperator*{\loc}{loc}

\newcommand{\e}{\varepsilon}

\renewcommand{\l}{\lambda}

\newcommand{\me}{\mathrm{e}}

\newcommand{\con}{\mathcal{C}}

\newcommand{\oT}{\mathbf{T}}

\RedeclareMathOperator*{\Im}{Im}
\RedeclareMathOperator*{\Re}{Re}
\renewcommand{\phi}{\varphi}

\renewcommand{\d}{\ \text{d}}

\newcommand{\dd}[2][ ]{\tfrac{\text{\normalfont d}#1}{\text{\normalfont d}#2}}

\makeatletter
\@ifundefined{ifCustomTheorems}{}
{
    \setcounter{section}{0}
    \counterwithout{equation}{section}
    \counterwithout{figure}{section}
    \newtheorem{definition}{Definition}[section]
    \theoremstyle{definition}
    \newtheorem{remark}[definition]{Remark}\Crefname{remark}{Remark}{Remarks}
    \newtheorem{algo}[definition]{Algorithm}\Crefname{algo}{Algorithm}{Algorithms}
    \Crefname{example}{Example}{Examples}
    \theoremstyle{plain}
    \Crefname{prop}{Proposition}{Propositions}
    \Crefname{corollary}{Corollary}{Corollaries}
    \Crefname{assertion}{Assertion}{Assertions}
    \newtheorem{theorem}[definition]{Theorem}\Crefname{theorem}{Theorem}{Theorems}
    \newtheorem{lemma}[definition]{Lemma}\Crefname{lemma}{Lemma}{Lemmata}
}
\makeatother

\begin{document}

\begin{frontmatter}
\title{Funnel MPC for nonlinear systems with arbitrary relative degree\tnoteref{funding}}

\tnotetext[funding]{This work was funded by the Deutsche Forschungsgemeinschaft (DFG, German Research Foundation)~-- Project-ID 471539468.}
\author[1]{Thomas Berger}\ead{thomas.berger@math.upb.de}
\author[1,2]{Dario Dennst\"adt}\ead{dario.dennstaedt@uni-paderborn.de}
\address[1]{Universit\"at Paderborn, Institut f\"ur Mathematik, Warburger Str.~100, 33098~Paderborn, Germany}
\address[2]{Technische Universit\"at Ilmenau, Institut f\"ur Mathematik, Weimarer Stra\ss e 25, 98693 Ilmenau, Germany}

\begin{abstract}
The model predictive control (MPC) scheme funnel MPC enables output tracking of
smooth reference signals with prescribed error bounds for nonlinear multi-input multi-output
systems with stable internal dynamics. Earlier works achieved the control objective for system with relative degree restricted to one or incorporated additional feasibility constraints in the optimal control problem. Here we resolve these limitations by introducing a modified stage cost function relying on a weighted sum of the tracking error derivatives. The weights need to be sufficiently large and we state explicit lower bounds. Under these assumptions we are able to prove initial and recursive feasibility of the novel funnel MPC scheme for systems with arbitrary relative degree~--
without requiring any terminal conditions, a sufficiently long prediction horizon or 
additional output constraints.
\end{abstract}

\begin{keyword}
model predictive control, funnel control, reference tracking,  nonlinear systems, initial feasibility, recursive feasibility
\end{keyword}

\end{frontmatter}

\section{Introduction}
Model predictive control (MPC) is a nowadays widely used control technique which has seen various
applications, see e.g.~\cite{QinBadg03} and also~\cite{samad2020industry}. It is applicable to nonlinear multi-input multi-output
system and able to take state and control constraints directly into account. MPC relies on the
iterative solution of finite horizon optimal control problems (OCP), see
e.g.~\cite{rawlings2017model,grune2017nonlinear}.

Solvability of the OCP at any particular time instance is essential for the successful application of MPC.
Incorporating suitably designed terminal conditions (costs and constraints) in the optimization problem
is an often used method to guarantee \textit{initial and recursive feasibility}, meaning guaranteeing 
that the solvability of the OCP at a particular instance in time automatically implies
that the OCP can be solved at the successor time instance.
However, the computational effort for solving the OCP and finding initially feasible control signals
becomes significantly more complicated by the introduction of such (artificial) terminal conditions.
Thus, the domain of admissible controls for MPC might shrink substantially, see
e.g.~\cite{chen2003terminal,gonzalez2009enlarging}. 
Alternative methods relying on controllability conditions, e.g. cost
controllability~\cite{CoroGrun20}, require a sufficiently long prediction horizon,
see e.g.~\cite{boccia2014stability,EsteWort20}. Especially in the presence of time-varying state and
output constraints these techniques are considerably more involved, see
e.g.~\cite{manrique2014mpc}.

Funnel MPC was proposed in~\cite{berger2019learningbased} to overcome these restrictions. It
allows for output reference tracking such that the tracking error evolves within predefined (time-varying)
performance bounds. 
This is different from the output tracking or output regulation problem on which
other MPC approaches usually focus, see e.g.~\cite{kohler2018nonlinear,limon2018nonlinear}.
Instead of prescribing a transient behavior for the system output, 
the control objective in these works is to achieve asymptotic tracking, i.e., ensuring the convergence of the tracking error to zero.
In~\cite{kohler2021constrained}, the output regulation problem is considered in the presence of time-invariant constraints. 
The constraint satisfaction is ensured by assuming suitable stabilizability and detectability conditions and a sufficiently long
prediction horizon. A similar control objective is pursued by tube-based MPC schemes, see e.g.~\cite{MaynSero05} for linear systems 
and~\cite{FaluMayn14,rakovic2022homothetic,KohlSolo21} for nonlinear systems.
Similar to funnel MPC, the tracking error is confined to a controllable and known range which might change over time.
The intention is to compensate for model uncertainties and disturbances acting on the system.
These methods ensure safe operation by introducing tightening tubes around the input and output constraints
in order to ensure robust satisfaction of the given constraints.
However, these tubes are usually calculated offline and cannot be arbitrarily chosen by the user since they have to encompass the system uncertainties~-- 
see e.g.~\cite{lopez2019dynamic}, where the tubes and the reference trajectory are simultaneously optimized depending on the proximity to the tube boundary.

While the first proposed funnel MPC algorithm in~\cite{berger2019learningbased} incorporated output constraints 
in the OCP, it was shown in the successor work~\cite{BergDenn21} that for a class of systems with
relative degree one and, in a certain sense, input-to-state stable internal dynamics, these
constraints are superfluous. Utilizing a ``funnel-like'' stage cost, which penalizes the tracking
error and becomes infinite when approaching predefined boundaries, guarantees initial and recursive
feasibility~-- without the necessity to impose additional terminal conditions or requirements on the
length of the prediction horizon.

Funnel MPC is inspired by funnel control which is an adaptive feedback control technique of high-gain
type first proposed in~\cite{IlchRyan02b}, see also the recent work~\cite{BergIlch21} for a
comprehensive literature overview.
The funnel controller is inherently robust and allows for output tracking with prescribed
performance guarantees for a fairly large class of systems solely invoking structural assumptions.
In contrast to MPC, funnel control does not use a model of the system. The control input signal is
solely determined by the instantaneous values of the system output.
The controller therefore cannot ``plan ahead''. This often results in unnecessary high control
values and a rapidly changing control signal with peaks.
Compared to this, by utilizing a system model, funnel MPC exhibits a significantly better controller
performance in numerical simulations, see~\cite{BergDenn21,berger2019learningbased}.
A direct combination of both control techniques which allows for the application of funnel MPC in the 
presence of disturbances and even a structural plant-model mismatch was recently proposed in~\cite{BergDenn23}.
This approach was further extended in~\cite{LanzaDenn23learning} by a learning component which realizes 
online learning of the model to allow for a steady improvement of the controller performance over time.

Nevertheless, the results of~\cite{BergDenn21,BergDenn23,LanzaDenn23learning} are still restricted to the case of systems with relative degree one.
Utilizing so-called feasibility constraints in the optimization problem and restricting the class of
admissible funnel functions, the case of arbitrary relative degree was considered in~\cite{BergDenn22}.
Like in previous results no terminal conditions  nor requirements on the length of the prediction horizon are imposed.
But then again these  feasibility constraints lead to an increased computational effort and they depend on a number of design parameters which are not easy to determine.
Furthermore, the cost functional used in~\cite{BergDenn22} is rather complex (using several auxiliary error variables).
In the present paper, we resolve these problems and propose a novel cost functional to extend funnel MPC to systems with arbitrary relative degree.
We further enlarge the considered system class considered in previous works to encompass systems with nonlinear time delays and potentially infinite-dimensional internal dynamics.
Similar to funnel MPC for relative degree one systems, only the distance of one error variable to the funnel boundary is penalized and no feasibility constraints are required.

\subsection{Nomenclature}
$\N$ and $\R$ denote natural and real numbers, respectively.
$\N_0:=\N\cup\{0\}$ and $\Rp:=[0,\infty)$.
$\Norm{x}:=\sqrt{\al x,x\ar}$~denotes the Euclidean norm of $x\in\R^n$.
$\Norm{A}$ denotes the induced operator norm
$\Norm{A}:=\sup_{\Norm{x} = 1}\Norm{Ax}$ for $A\in\R^{n\times n}$.
$\GL_n(\R)$ is the group of invertible $\R^{n\times n}$ matrices.
For $V\subset\R^m$, $\con(V,\R^n)$ is the linear space of continuous functions $f:V\to\R^n$
and $\con^p(V,\R^n)$ is the linear space of $p$-times continuously differentiable functions,
where $p\in\N\cup \{\infty\}$;  $\con^0(V,\R^n):=\con(V,\R^n)$.
On an interval $I\subset\R$,  $L^\infty(I,\R^n)$ denotes the space of measurable and essentially bounded
functions $f: I\to\R^n$ with norm $\SNorm{f}:=\esssup_{t\in I}\Norm{f(t)}$,
$L^\infty_{\text{loc}}(I,\R^n)$ the set of measurable and locally essentially bounded functions, and $L^p(I,\R^n)$
the space of measurable and $p$-integrable functions with norm $\LNorm[p]{\cdot}$ and with $p\ge 1$.
Furthermore, $W^{k,\infty}(I,\R^n)$ is the Sobolev space of all $k$-times weakly differentiable functions
$f:I\to\R^n$ such that $f,\dots, f^{(k)}\in L^{\infty}(I,\R^n)$;  and $f|_J$ denotes the restriction of a function $f: I \to \R^n$ to the interval $J \subseteq I$.

\subsection{System class}

We consider nonlinear control affine multi-input multi-output systems of order $r\in\N$ of the form
\begin{equation} \label{eq:Sys}
    \begin{aligned}
   & y^{(r)}(t) = f \big(\oT(y,\ldots,y^{(r-1)} )(t) \big) \\
                &\phantom{=}\qquad\quad+ g \big(\oT(y,\ldots,y^{(r-1)} )(t) \big) u(t), \\
   &\left. \begin{aligned}
       y|_{[-\sigma,0]} = y^0  \in \con^{r-1}([-\sigma,0],\R^m), && \mbox{if } \sigma >0,\\
       \big(y(0),\ldots,y^{(r-1)}(0)\big) = y^0  \in\R^{rm}, && \mbox{if } \sigma = 0,
   \end{aligned} \right\}
    \end{aligned}
\end{equation}
with initial ``memory'' $\sigma \ge 0$, functions $f \in \con(\R^q, \R^m)$,  $g \in \con(\R^q, \R^{m \times m})$, and an operator $\oT$.
The operator~$\oT$ is causal, locally Lipschitz, and satisfies a bounded-input bounded-output and a limited memory property.
It is characterised in detail in the following~\Cref{Def:OperatorClass}.

\begin{definition} \label{Def:OperatorClass} 
For $n,q\in\N$ and $\sigma\geq 0$, the set $\cT_{\sigma}^{n,q}$ denotes the class of operators $\textbf{T}:
\con([-\sigma,\infty),\R^n) \to L^\infty_{\loc} (\Rp, \R^{q})$
for which the following properties hold:
\begin{itemize}
    \item\textit{Causality}:  $\fa y_1,y_2\in\con([-\sigma,\infty),\R^n)$  $\fa t\geq 0$:
    \[
        y_1\vert_{[-\sigma,t]} = y_2\vert_{[-\sigma,t]}
        \ \Impl\ 
        \textbf{T}(y_1)\vert_{[0,t]}=\textbf{T}(y_2)\vert_{[0,t]}.
    \]
    \item\textit{Local Lipschitz}: 
    $\fa t \ge 0 $ $\fa y \in \con([-\sigma,t] , \R^n)$ 
    $\ex \Delta, \delta, c > 0$ 
    $\fa y_1, y_2 \in \con([-\sigma,\infty) , \R^n)$ with
    $y_1|_{[-\sigma,t]} = y_2|_{[-\sigma,t]} = y $ 
    and $\Norm{y_1(s) - y(t)} < \delta$,  $\Norm{y_2(s) - y(t)} < \delta $ for all $s \in [t,t+\Delta]$:
    \[
     \hspace*{-2mm}   \esssup_{\mathclap{s \in [t,t+\Delta]}}  \Norm{\textbf{T}(y_1)(s) \!-\! \textbf{T}(y_2)(s) }  
        \!\le\! c \ \sup_{\mathclap{s \in [t,t+\Delta]}}\ \Norm{y_1(s)\!-\! y_2(s)}\!.
    \] 
    \item\textit{Bounded-input bounded-output (BIBO)}:
    $\fa c_0 > 0$ $\ex c_1>0$  $\fa y \in \con([-\sigma,\infty), \R^n)$:
    \[
    \sup_{t \in [-\sigma,\infty)} \Norm{y(t)} \le c_0 \ 
    \Impl \ \sup_{t \in [0,\infty)} \Norm{\textbf{T}(y)(t)}  \le c_1.
    \]
    \item\textit{Limited memory}: 
    $\ex \tau\geq0$ $\fa t\geq 0$
    $\fa y_1, y_2 \in \con([-\sigma,\infty) , \R^n)$ with 
    $y_1|_{I}= y_2|_{I}$ on the interval $I:= [t-\tau,\infty)\cap[-\sigma,\infty)$ and $\oT(y_1)|_{J} = \oT(y_2)|_{J}$ on the interval $J:=[t-\tau,t]\cap[0,t]$:
    \[
        \oT(y_1)\vert_{[t,\infty)}=\oT(y_2)\vert_{[t,\infty)}.
    \]
    The number $\tau$ in the above property is called \textit{memory limit} of the operator~$\oT$.
\end{itemize}
\end{definition}
Note that many physical phenomena such as \emph{backlash} and \emph{relay hysteresis}, and \emph{nonlinear time delays}
can be modelled by means of the operator~$\oT$, where $\sigma$ corresponds to the initial delay, cf.~\cite[Sec.~1.2]{BergIlch21}.
Moreover, systems with infinite-dimensional internal dynamics can be represented by~\eqref{eq:Sys}, see~\cite{BergPuch20}.
For a practically relevant example of infinite-dimensional internal dynamics (modelled by an operator~$\oT$) we refer to the moving water tank system considered in~\cite{BergPuch22}.
Compared to previous works, the limited memory property is new here and required in the context of MPC,
in order to ensure that in each MPC step only the history of the state up to the memory limit $\tau\ge 0$ is utilized,
instead of requiring the full history, which would be infeasible in practice.
In Section~\ref{Sec:Sim}, we illustrate how the limited memory property can be checked by means of a specific example.

For given $\hat{t}\geq 0$, $\tau,\sigma\geq0$, we use the notation $I_{\sigma}^{\hat{t},\tau}:=[-\sigma,\hat{t}]\cap[\hat{t}-\tau,\hat{t}]$.
For a control $u\in L^\infty_{\loc}([\hat{t},\infty), \R^m)$,
a continuous function $x = (x_1,\ldots,x_r)$ with $x_i: [-\sigma,\omega) \to \R^{m}$, $\omega \in(\hat{t},\infty]$, $i=1,\ldots,r$,
is called a solution of the initial value problem~\eqref{eq:Sys} 
(in the sense of \textit{Carath\'{e}odory})
at initial time $\hat{t}$ and with initial data $\hat{y}\in\con^{r-1}(I_{\sigma}^{\hat{t},\tau},\R^m)$ and
$\hat{\oT}\in L^\infty_{\loc} (I_{0}^{\hat{t},\tau}, \R^{q})$,  if 
\begin{equation} \label{eq:InitialValue}
\begin{aligned}
    x|_{I_{\sigma}^{\hat{t},\tau}} &= (\hat{y},\hat{y}^{(1)},\ldots,\hat{y}^{(r-1)}),\\
    \oT(x)|_{I_{0}^{\hat{t},\tau}} &=\hat{\oT},
\end{aligned}
\end{equation}
and $x\vert_{[\hat{t},\omega)}$ is absolutely continuous such that $\dot x_i(t) = x_{i+1}(t)$ for $i=1,\ldots,r-1$, and $\dot x_r(t) = f(\oT(x)(t)) + g(\oT(x)(t)) u(t)$ for almost all~$t\in[\hat{t},\omega)$.
A solution $x$ is said to be \textit{maximal}, if it has no right extension that is also a solution.
This maximal solution is called the \textit{response} associated with $u$ and denoted by~$x(\cdot;\hat{t},\hat{y},\hat{\oT},u)$. 
Its first component $x_1$ is  denoted by~$y(\cdot;\hat{t},\hat{y},\hat{\oT},u)$.
Note that in the above definition we did not distinguish between the cases $\sigma>0$ and $\sigma=0$ as in~\eqref{eq:Sys},
since a larger variety of cases is possible here.
Essentially, the cases $I_{\sigma}^{\hat{t},\tau}\neq \{\hat t\}$ and $I_{\sigma}^{\hat{t},\tau} = \{\hat t\}$ mus be distinguished;
in the latter case, we will implicitly assume a representation of the initial condition~$\hat y$ as in~\eqref{eq:Sys} throughout the article.

Let us provide an additional explanation of the above definition.
At first glance it might seem that $\hat \oT$ is completely determined by $(\hat{y},\hat{y}^{(1)},\ldots,\hat{y}^{(r-1)})$.
However, the latter is not an element of the domain of $\oT$ in general.
In fact, the second condition in~\eqref{eq:InitialValue} entwines the full history of~$x$ (by causality, $T(x)(t)$ 
is defined in terms of $x|_{[-\sigma,t]}$ for any $t\ge 0$) with the initial datum~$\hat \oT$.
In contrast to this, the first condition in~\eqref{eq:InitialValue} only fixes the history of~$x$ on the interval $I_{\sigma}^{\hat{t},\tau}$.

We summarize our assumptions and define the general system class under consideration.
\begin{definition}[System class]\label{Def:SysClass}
    We say that the system~\eqref{eq:Sys} belongs to the \emph{system class}~$\cN^{m,r}$ for $m,r\in\N$,
    written~$(f,g,\oT)\in\cN^{m,r}$, if, for some $q\in\N$ and $\sigma \geq0$, the following holds:
     $f\in\con(\R^q ,\R^m)$,
     $g\in \con(\R^q ,\R^{m\times m})$ satisfies $g(x)\in\GL_{m}(\R)$ for all $x\in\R^q$,
     and ${\oT\in\cT^{rm,q}_{\sigma}}$.
\end{definition}

\subsection{Control objective}
The objective is to design a control strategy that allows tracking of a given reference
trajectory~$y_{\rf}\in W^{r,\infty}(\Rp,\R^{m})$ within pre-specified error bounds. To be more
precise, the tracking error ~$e(t):=y(t)-y_{\rf}(t)$  shall evolve within the prescribed performance funnel
\begin{align*}
    \cF_\psi= \setdef{(t,e)\in \Rp\times\R^{m}}{\Norm{e} < \psi(t)}.
\end{align*}
This funnel is determined by the choice of the function~$\psi$ belonging  to
\begin{align*}
    \cG:=\setdef
        {\psi\in W^{1,\infty}(\Rp,\R)}
        {
          \begin{array}{l} \inf_{t\ge 0}  \psi(t) > 0,\\
          \exists\, \alpha, \beta>0\ \forall\, t\ge 0:\\
          \dot \psi(t) \ge -\alpha \psi(t) + \beta 
          \end{array}
        },
\end{align*}
see also Figure~\ref{Fig:funnel}.
 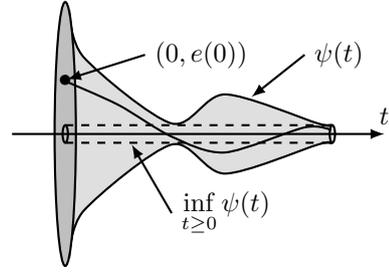
\begin{figure}[h]
  \begin{center}
\begin{tikzpicture}[scale=0.35]
\tikzset{>=latex}
  \filldraw[color=gray!25] plot[smooth] coordinates {(0.15,4.7)(0.7,2.9)(4,0.4)(6,1.5)(9.5,0.4)(10,0.333)(10.01,0.331)(10.041,0.3) (10.041,-0.3)(10.01,-0.331)(10,-0.333)(9.5,-0.4)(6,-1.5)(4,-0.4)(0.7,-2.9)(0.15,-4.7)};
  \draw[thick] plot[smooth] coordinates {(0.15,4.7)(0.7,2.9)(4,0.4)(6,1.5)(9.5,0.4)(10,0.333)(10.01,0.331)(10.041,0.3)};
  \draw[thick] plot[smooth] coordinates {(10.041,-0.3)(10.01,-0.331)(10,-0.333)(9.5,-0.4)(6,-1.5)(4,-0.4)(0.7,-2.9)(0.15,-4.7)};
  \draw[thick,fill=lightgray] (0,0) ellipse (0.4 and 5);
  \draw[thick] (0,0) ellipse (0.1 and 0.333);
  \draw[thick,fill=gray!25] (10.041,0) ellipse (0.1 and 0.333);
  \draw[thick] plot[smooth] coordinates {(0,2)(2,1.1)(4,-0.1)(6,-0.7)(9,0.25)(10,0.15)};
  \draw[thick,->] (-2,0)--(12,0) node[right,above]{\normalsize$t$};
  \draw[thick,dashed](0,0.333)--(10,0.333);
  \draw[thick,dashed](0,-0.333)--(10,-0.333);
  \node [black] at (0,2) {\textbullet};
  \draw[->,thick](4,-3)node[right]{\normalsize$\inf\limits_{t \ge 0} \psi(t)$}--(2.5,-0.4);
  \draw[->,thick](3,3)node[right]{\normalsize$(0,e(0))$}--(0.07,2.07);
  \draw[->,thick](9,3)node[right]{\normalsize$\psi(t)$}--(7,1.4);
\end{tikzpicture}
\end{center}
 \vspace*{-4mm}
 \caption{Error evolution in a funnel $\mathcal F_{\psi}$ with boundary $\psi(t)$.}
 \label{Fig:funnel}
 \vspace*{-4mm}
 \end{figure}
Note that the evolution in $\mathcal{F}_{\psi}$ does not force the tracking error to converge to zero asymptotically. 
Furthermore, the funnel boundary is not necessarily monotonically decreasing 
and there are situations, like in the presence of periodic disturbances,
where widening the funnel over some later time interval might be beneficial. 
The specific application usually dictates the constraints on the tracking error and thus
indicates suitable choices for~$\psi$.

To achieve the control objective, we introduce auxiliary error variables.
Define, for $(\xi_1,\ldots,\xi_r)\in\R^{rm}$ with $\xi_i\in\R^m$ and for parameters $k_1,\ldots, k_{r-1}\in\Rp$, the functions $e_i:\R^{rm}\to\R^m$ recursively by
\begin{equation} \label{eq:ErrorVar}
    \begin{aligned}
        e_1 (\xi_1,\ldots,\xi_r)&:= \xi_1,\\
        e_{i+1} (\xi_1,\ldots,\xi_r)&:=e_{i}(S(\xi_1,\ldots,\xi_r))+k_i e_{i}(\xi_1,\ldots,\xi_r),
    \end{aligned}
\end{equation}%
for $i=1,\ldots,r-1$, where 
\[
    S:\R^{rm}\to\R^{rm},\ S(\xi_1,\ldots,\xi_r):= (\xi_2,\ldots,\xi_r,0)
\]
is the left shift operator.

\begin{remark}
Using the shorthand notation 
\begin{equation*}
    \chi(\zeta)(t):=(\zeta(t),\dot{\zeta}(t),\ldots,\zeta^{(r-1)}(t))\in\R^{rm}    
\end{equation*}
for a function  $\zeta\in W^{r,\infty}(\Rp,\R^m)$ and $t\in\Rp$, we get 
\begin{equation} \label{eq:ErrorVarDyn}
    \begin{aligned}
    e_1    (\chi(\zeta)(t))&=\zeta(t),\\
    e_{i+1}(\chi(\zeta)(t))&=\dd{t}e_{i}(\chi(\zeta)(t))+k_i e_{i}(\chi(\zeta)(t))
    \end{aligned}
\end{equation}
for $i= 1,\ldots,r-1$.
Furthermore, using the polynomials $p_i(s)=\prod_{j=1}^i(s+k_j)\in\R[s]$, 
the function $e_{i+1}(\chi(\zeta)(t))$ can be represented as
\[
    e_{i+1}(\chi(\zeta)(t))=p_i(\dd{t})\zeta(t)
\]
for $i=1,\ldots, r-1$.
In~\Cref{Algo:FMPC} below, the error variables $e_i$ are evaluated at $\chi(y-y_{\rf})(t)$.
Note that $e_1(\chi(y-y_{\rf})(t))$ is identical to the tracking error $e(t)=y(t)-y_{\rf}(t)$.
\end{remark}

\section{Funnel MPC}
We propose for $\theta\in\cG$, 
design parameter~$\lambda_u\in\Rp$, 
and functions $e_1,\ldots,e_r$ as defined in~\eqref{eq:ErrorVar} with parameters~$k_i>0$ for $i=1,\ldots, r-1$,
the \textit{stage cost function } $\ell_\theta:\Rp\times\R^{rm}\times\R^{m}\to\R\cup\{\infty\}$ defined by
\begin{align}\label{eq:stageCostFunnelMPC}
    \ell_\theta(t,\xi,u) \!=\!
    \begin{dcases}
        \frac {\Norm{e_r(\xi)}^2}{\theta(t)^2 
        - \Norm{e_r(\xi)}^2} \!+\! \l_u \Norm{u}^2\!\!, & \!\!\Norm{e_r(\xi)} \!\neq \theta(t)\\
        \infty,                                 & \!\! \text{else}.
    \end{dcases}
\end{align}
\begin{algo}[Funnel MPC]\label{Algo:FMPC}\ \\
    \textbf{Given:} System~\eqref{eq:Sys}, reference signal $y_{\rf}\in
    W^{r,\infty}(\Rp,\R^{m})$, funnel function $\theta\in\cG$, input saturation level $M>0$, memory limit $\tau\geq0$ of the operator~$\oT$,
    initial data $y^0\in\con^{r-1}([-\sigma,0],\R^m)$ 
    and stage cost function~$\ell_\theta$ as in~\eqref{eq:stageCostFunnelMPC}.\\
    \textbf{Set} the time shift $\delta >0$, the prediction horizon $T\geq\delta$, and initialize
    the current time~$\hat{t}:=0$.\\
    \textbf{Steps:}
    \begin{enumerate}[label=(\alph*), ref=\alph*, leftmargin=*]
    \item\label{agostep:FMPCFirst}
    Obtain a measurement of the output $y$ of~\eqref{eq:Sys} and of~$\oT(y)$ on the interval $I_{\sigma}^{\hat{t},\tau}:=[-\sigma,\hat{t}]\cap[\hat{t}-\tau,\hat{t}]$
    and set $\hat y := y|_{I_{\sigma}^{\hat{t},\tau}}$ and  $\hat{\oT}:=\oT(\chi(y))|_{I_{0}^{\hat{t},\tau}}$.
    \item Compute a solution $u^{\star}\in L^\infty([\hat{t},\hat{t} +T],\R^{m})$ of
    \begin{equation}\label{eq:FMPC_OCP}
            \!\!\!\mathop
            {\operatorname{minimize}}_{\mathcal{\substack{u\in L^{\infty}([\hat{t},\hat{t}+T],\R^{m}\!),\\\SNorm{u}\leq M}}}\!
            \int_{\hat{t}}^{\hat{t}+T}\!\!\!\!\!\!\!\!\ell_\theta(t,\zeta(t),u(t))\!\d t ,
    \end{equation}
    where $\zeta(t):=\chi(y(\cdot;\hat{t},\hat{y},\hat{\oT},u)-y_{\rf})(t)$.
    \item Apply the time-varying feedback law $\mu:[\hat{t},\hat{t}+\delta)\times\con^{r-1}(I_{\sigma}^{\hat{t},\tau},\R^m)\times  L^\infty_{\loc} (I_{0}^{\hat{t},\tau}, \R^{q})\to\R^m$ \
    defined by
        \begin{equation}\label{eq:FMPC-fb}
            \ \mu(t,\hat y,\hat{\oT}) =u^{\star}(t),
        \end{equation}
        to system~\eqref{eq:Sys}. Increase $\hat{t}$ by $\delta$ and go to Step~\eqref{agostep:FMPCFirst}.
    \end{enumerate}
\end{algo}

\begin{remark}
   For a nonlinear system of the form
   \begin{equation}\label{eq:SysWithState}
   \begin{aligned}
       \dot{x}(t)  & = \tilde f(x(t)) + \tilde g(x(t)) u(t),\quad x(0)=x^0\in\R^n,\\
       y(t)        & = h(x(t)),
   \end{aligned} 
   \end{equation}
   with  nonlinear functions~$\tilde f:\R^n\to \R^n$,
   $\tilde g:\R^n\to \R^{n\times m}$ and $h : \R^n \to \R^m$, 
   there exists, under assumptions provided in~\cite[Cor.~5.6]{ByrnIsid91a},
   a coordinate transformation induced by a diffeomorphism~$\Phi:\R^n\to\R^n$
   which puts the system in the form~\eqref{eq:Sys} with $\sigma=0$, appropriate functions~$f$ and~$g$ and an operator $\textbf{T}$, which is the solution operator of the internal dynamics of the transformed system.
   Assuming the existence of the diffeomorphism~$\Phi$, the~funnel MPC~\Cref{Algo:FMPC} can be directly applied to the system~\eqref{eq:SysWithState} without computing~$\Phi$. In this case, the output derivatives $\dot y,\ldots, y^{(r-1)}$ required in the OCP~\eqref{eq:FMPC_OCP} can be determined as functions of the state; e.g., $\dot y(t) = h'(x(t)) \tilde f(x(t))$ if $r=1$.
   All results presented in this paper can also be expressed for the system~\eqref{eq:SysWithState} using~$\Phi$.
   Concrete knowledge about the coordinate transformation however is not required for 
   the design and application of the controller -- it is merely used as a tool for the proofs.
\end{remark}

In the following main result we show that for a funnel function $\psi\in\cG$, a reference signal $y_{\rf}$ and sufficiently large $k_1,\ldots,k_{r-1}$ (depending on the choice of~$\psi$, $y_{\rf}$ and the initial data~$y^0$) there exists a sufficiently large saturation level $M>0$ such that the funnel MPC Algorithm~\ref{Algo:FMPC} (with a suitable function $\theta\in\cG$) is initially
and recursively feasible for every prediction horizon $T > 0$ and that it guarantees the evolution of the tracking error
within the performance funnel~$\cF_\psi$.

\begin{theorem}\label{Th:FunnelMPC}
    Consider system~\eqref{eq:Sys} with $(f,g,\oT)\in\cN^{m,r}$, initial data $y^0  \in \con^{r-1}([-\sigma,0],\R^m)$ and let $\tau\ge 0$ be the memory limit of the operator~$\oT$.
    Let $y_{\rf}\in W^{r,\infty}(\Rp,\R^{m})$ and choose $\psi\in\cG$ with associated constants $\alpha, \beta >0$.
    In addition, let $\gamma\in(0,1)$ such that
    \[
        \Norm{y^0(0)-y_{\rf}(0)} \le \gamma^r \psi(0).
    \]
    Furthermore, choose parameters $k_1,\ldots,k_{r-1}$ such that for all $i=2,\ldots,r-1$ we have
    \begin{equation}\label{eq:cond-k_i}
    \begin{aligned}
       k_1 &\ge \frac{2\Norm{(\dot y^0-\dot y_{\rf})(0)}}{\gamma^{r-1}(1-\gamma)\psi(0)} +\frac{2\left(\alpha+\frac{1}{\gamma^{r-1}}\right)}{1-\gamma},\\
        k_i &\ge \frac{2\gamma\Norm{\dd{t} e_{i}(\chi(y^0-y_{\rf})(0))}}{(1-\gamma)\big(\Norm{e_{i}(\chi(y^0-y_{\rf})(0))}+\frac{\beta}{\alpha \gamma^{i-2}}\big)} + \frac{2(1+\alpha)}{1-\gamma}.
    \end{aligned}
    \end{equation}
    \noindent
    Then, there exists $M>0$ such that the funnel MPC~\Cref{Algo:FMPC} with prediction horizon $T>0$, time shift $\delta>0$,
    and stage cost function~$\ell_\theta$ with
    \begin{multline}\label{eq:DefTheta}
        \theta(t):= \frac{1}{\gamma}\big( \Norm{\dd{t} e_{r-1}(\chi(y^0-y_{\rf})(0))} \\
         + k_{r-1} \Norm{e_{r-1}(\chi(y^0-y_{\rf})(0))}\big) \me^{-\alpha t} + \frac{\beta}{\alpha\gamma^{r-1}}
    \end{multline}
    for $r>1$ and $\theta(t):=\psi(t)$ for $r=1$,
     is initially and recursively feasible, i.e., at time $\hat t = 0$ and at each
    successor time $\hat t\in \delta\N$ the OCP~\eqref{eq:FMPC_OCP}
    has a solution.
    In particular, the closed-loop system consisting of~\eqref{eq:Sys} and the funnel MPC feedback~\eqref{eq:FMPC-fb}
    has a (not necessarily unique) global solution $x:[-\sigma,\infty)\to\R^{rm}$ with corresponding output $y = x_1$ and the corresponding input is given by
    \[
        u_{\rm FMPC}(t) = \mu(t,y|_{I_{\sigma}^{\hat{t},\tau}},\oT(\chi(y))|_{I_{0}^{\hat{t},\tau}}).
    \]
    for $t\in [\hat t,\hat t+\delta)$, $\hat t\in t^0+\delta\N$, and $I_{\sigma}^{\hat{t},\tau}=[-\sigma,\hat{t}]\cap[\hat{t}-\tau,\hat{t}]$.
    Furthermore, each global solution~$x$ with corresponding output~$y$ and input $u_{\rm FMPC}$ satisfies:
    \begin{enumerate}[(i)]
        \item\label{th:item:BoundedInput}
            $\fa t\ge t^0:\quad \Norm{u_{\rm FMPC}(t)}\leq M$,
        \item\label{th:item:ErrorInFunnel}$\fa t\geq 0:\quad \Norm{y(t)-y_{\rf}(t)} <
            \psi(t)$.
    \end{enumerate}
\end{theorem}

Note that for $r=1$ the above result essentially coincides with~\cite[Thm.~2.10]{BergDenn21},
except for the different system classes considered.
In particular, the stage cost function in~\eqref{eq:stageCostFunnelMPC} is the same as in~\cite{BergDenn21} in this case.
In this sense, the findings of the present paper represent an extension of the results of~\cite{BergDenn21}.

\section{Proof of the main result}
Throughout this section, let the assumptions of \Cref{Th:FunnelMPC} hold. 
Then set $e_i^0 :=e_{i}(\chi(y^0-y_{\rf})(0))$ and $\dot e_i^0 := \dd{t} e_{i}(\chi(y^0-y_{\rf})(0))$ for $i=1,\ldots,r-1$ and we define $\psi_1:=\psi$ and $\psi_2,\ldots,\psi_r$ as follows:
\begin{equation}\label{eq:psi_i}
    \psi_{i+1}(t):= \frac{1}{\gamma^{r-i}}\big(\Norm{\dot e_i^0} + k_i \Norm{e_i^0}\big) \me^{-\alpha t} + \frac{\beta}{\alpha\gamma^{r-1}}
\end{equation}
for $t\ge 0$ and $i=1,\ldots,r-1$.
Note that $\psi_i\in\cG$ for all $i=1,\ldots, r$.
Further note that $\psi_r =\theta$ as in~\eqref{eq:DefTheta}.
In order to achieve that the tracking error~$e=y-y_{\rf}$ evolves within the funnel $\mathcal{F}_{\psi}$, 
we address the problem of ensuring that, for all $t\ge 0$, $\chi(e)(t)$ is an element of the set
\[
    \cD_{t}^{r}:=\setdef
        { \xi\in\R^{rm}}
        {
              \Norm{e_{i}(\xi)}<\psi_i(t),\ i=1,\ldots, r
        }.
\]
By construction of~$\psi_i$ and~\eqref{eq:ErrorVarDyn} we have
\[
    \|e_i(0)\| \le \|\dot e_{i-1}(0)\| + k_{i-1} \|e_{i-1}(0)\| < \psi_i(0)
\]
for all $i=2,\ldots,r$, and by assumption we have
\[
    \Norm{e(0)} \le \gamma^r \psi(0) < \psi_1(0).
\]
Therefore, $\chi(y^0-y_{\rf})(0)\in \cD_{0}^{r}$.
We define the set of all functions $\zeta\in\con^{r-1}([-\sigma,\infty),\R^m)$ 
which coincide with $y^0$ on the interval~$[-\sigma,0]$ and for which 
$\chi(\zeta-y_{\rf})(t)\in\cD^r_t$ on the interval $J^{s}:=[0,s)$ for some ${s\in(0,\infty]}$ 
as follows: 
\[
    \cY^r_s\!:=\!\!\setdef
        {\!\!\!\zeta\!\in\! \con^{r-1}([-\sigma,\!\infty),\R^m)\!\!\!}
        {\!\!\!\!
        \begin{array}{l}
             \chi(\zeta|_{[-\sigma,0]})=\chi(y^0), \\ \!\! \fa t\!\in\! J^s\!\!: \chi(\zeta\!-\!y_{\rf})(t)\!\in\!\cD_t^r
        \end{array}\!\!\!\!\!
        }\!.
\] 

\begin{lemma}\label{Lemma:DynamicBounded}
    Consider the system~$\eqref{eq:Sys}$ with $(f,g,\oT)\in\cN^{m,r}$.
    Let $\psi_i\in\cG$, for $i=1,\ldots, r$ with parameters $k_i>0$ for $i=1,\ldots, r-1$.
    Further, let $y_{\rf}\in W^{r,\infty}(\Rp,\R^m)$ and 
    $y^0  \in \con^{r-1}([-\sigma,0],\R^m)$,
    Then, there exist constants~$f_{\max}$, $g_{\max}>0$ such that for all $s\in(0,\infty]$, $\zeta\in\cY^r_s$, and $t\in[0,s)$
    \begin{align*}
        f_{\max}&\geq\SNorm{f(\oT(\chi(\zeta))|_{[0,s)})},\\
        g_{\max}&\geq\SNorm{g(\oT(\chi(\zeta))|_{[0,s)})^{-1}}.
    \end{align*}
\end{lemma}
\begin{proof}
    We prove the Lemma by adapting the proof of~\cite[Lem.~2.2]{LanzaDenn23sampled} to the given setting. 
    By definition of~$\cY_\infty^r$ and $\cD_{t}^r$, we have for all $i=1,\ldots,r$
    \[
        \fa\zeta\in\cY_\infty^r\fa t\geq 0:\quad\Norm{e_i(\chi(\zeta-y_{\rf})(t))} < \psi_i(t).
    \]
    Due to the definition of the error variables~$e_i$ there exists an invertible matrix $S\in\R^{rm\times rm}$ such that
    \begin{equation}\label{eq:reps_ei_chi}
        \begin{pmatrix} e_1(\chi(\zeta-y_{\rf})) \\ \vdots \\  e_r(\chi(\zeta-y_{\rf}))\end{pmatrix} = S \chi(\zeta-y_{\rf}).
    \end{equation}
    Hence, by boundedness of $\psi_i$ and $y_{\rf}^{(i)}$ for all~$i=1,\ldots, r$,
    there exists a compact set $K\subset\R^{rm}$ with 
    \[
        \fa \zeta\in\cY_\infty^r\fa t\geq 0:\quad \chi(\zeta)(t)\in K.
    \]
    Invoking the BIBO property of the operator~$\oT$, there exists a compact set $K_q\subset\R^q$ with  
    $\oT(\xi)([0,\infty))\subset K_q$ for all $\xi\in\con(\Rp,\R^{rm})$ with $\xi([0,\infty))\subset K$.
    For arbitrary~$s\in(0,\infty)$ and $\zeta\in\cY_{s}^r$, we have $\chi(\zeta)(t)\in K$ for all $t\in[0,s)$.
    For every element $\zeta\in\cY_{s}^r$ the function $\chi(\zeta)|_{[-\sigma,s)}$ can be smoothly extended to a function~$\tilde{\zeta}\in\con([-\sigma,\infty),\R^m)^r$ 
    with $\tilde{\zeta}(t)\in K$ for all $t\in[0,\infty)$.
    We have $\oT(\tilde{\zeta})(t)\in K_q$ for all $t\in\Rp$ because of the BIBO property of the operator~$\oT$. 
    This implies $\oT(\chi(\zeta))|_{[0,s)}\in K_q$ for all $t\in [0,s)$ and $\zeta\in\cY_{s}^r$ since $\oT$ is causal.
    Since~$f(\cdot)$ and $g(\cdot)^{-1}$ are continuous, the constants $f_{\max}=\max_{x\in K_q}\Norm{f(x)}$ and $g_{\max}=\max_{x\in K_q}\Norm{g(x)^{-1}}$ are well-defined.
    For all $s\in(0,\infty]$ and $\zeta\in\cY_{s}^r$ we have 
    \[
        \fa t\in [0,s):\ \oT(\chi(\zeta))(t)\in K_q,
    \]
    which proves the assertion.
\end{proof}

\begin{lemma}\label{Lemma:Errors}
Under the assumptions of Theorem~\ref{Th:FunnelMPC}, consider the functions $\psi_2,\ldots,\psi_r\in\cG$ defined in~\eqref{eq:psi_i}.
Let $\hat t\ge 0$ and $\zeta\in \con^{r-1}([\hat t,\infty),\R^m)$ be such that $\chi(\zeta)(\hat t)\in\cD_{\hat t}^r$. If  
$\Norm{e_r(\chi(\zeta)(t))}< \psi_r(t)$ for all $t\in[\hat t,s)$ for some $s>\hat t$, then 
$\chi(\zeta)(t)\in\cD_{t}^r$ for all $t\in[\hat t,s)$.
\end{lemma}
\begin{proof}
Seeking a contradiction, we assume that for at least one ${i \in \{1,\ldots,r-1\}}$ there exists $t \in (\hat t,s)$
such that $\Norm{e_i(\chi(\zeta)(t))} \geq \psi_i(t)$. W.l.o.g.\ let~$i$ be the largest index with this property.
In the following we use the shorthand notation $e_i(t):=e_i(\chi(\zeta)(t))$ and $e(t) = e_1(t)$.
However, we like to emphasize that $e_i(0) \neq e_i^0$ (if $e_i$ is defined at~$0$) in general, since $\chi(\zeta)(0)\neq \chi(y^0)(0)$ is possible.
Invoking $\Norm{e_i(\hat t)}<\psi_i(\hat t)$ and continuity of the involved functions, define 
$t^\star := \min \setdef{ t \in [\hat t,s) }{ \Norm{e_i(t) } = \psi_i(t)}$.
Set $\e := \sqrt{\tfrac12 (1+\gamma)} \in (0,1)$.
Due to continuity there exists $t_\star := \max \setdef{ t \in [\hat t,t^\star) }{ \Norm{\tfrac{e_i(t)}{\psi_{i}(t)} } = \e}$, hence we have that $\e\le\Norm{\tfrac{e_i(t)}{\psi_{i}(t)} } \le 1$ for all $t\in[t_\star,t^\star]$.
Utilizing~\eqref{eq:ErrorVarDyn} and omitting the dependency on $t$, we calculate for $t \in [t_{\star},t^\star]$:
\begin{align*}
    \tfrac{1}{2}\dd{t}\Norm{\frac{e_i}{\psi_i}}^2
    &= \al\frac{e_i}{\psi_i},\frac{\dot{e}_i\psi_i-e_i\dot{\psi}_i}{\psi_i^2}\ar\\
    &= \al\frac{e_i}{\psi_i},-\rbl k_i+\frac{\dot{\psi}_i}{\psi_i}\rbr\frac{e_{i}}{\psi_i}+\frac{e_{i+1}}{\psi_i}\ar\\
    &\le -\rbl k_i+\frac{\dot{\psi}_i}{\psi_i}\rbr\Norm{\frac{e_i}{\psi_i}}^2 + \Norm{\frac{e_i}{\psi_i}} \frac{\Norm{e_{i+1}}}{\psi_i}\\
    &\leq -\rbl k_i+\frac{\dot{\psi}_i}{\psi_i}\rbr\e^2 +\frac{\psi_{i+1}}{\psi_i},
\end{align*}
where we used $\Norm{e_{i+1}(t)}\leq \psi_{i+1}(t)$ due to the maximality of $i$.
Now we distinguish the two cases $i=1$ and $i>1$. For $i=1$ we find that $\psi_1 = \psi$ and by properties of~$\cG$ it follows
\[
    -\frac{\dot{\psi}(t)}{\psi(t)}\le \frac{\alpha \psi(t) - \beta}{\psi(t)} \le \alpha.
\]
Furthermore, we have that $\psi(t) \ge \psi(0) \me^{-\alpha t} + \frac{\beta}{\alpha}$ for all $t\ge 0$. Therefore,
\begin{align*}
    &\frac{\psi_2(t)}{\psi(t)}
    \le  \frac{1}{\gamma^{r-1}} \frac{\big(\Norm{\dot e_1^0} + k_1 \Norm{e_1^0}\big) \me^{-\alpha t}}{\psi(0) \me^{-\alpha t} + \frac{\beta}{\alpha}}\\
    &\phantom{\frac{\psi_2(t)}{\psi(t)}= } + \frac{\beta}{\alpha\gamma^{r-1}\big(\psi(0) \me^{-\alpha t} + \frac{\beta}{\alpha}\big)} \\
    &\le\! \frac{1}{\gamma^{r-1}} \frac{\Norm{\dot e_1^0} + k_1 \Norm{e_1^0}}{\psi(0)} + \frac{1}{\gamma^{r-1}}
    \!\le\! \gamma k_1 \!+\! \frac{\Norm{\dot e_1^0}}{\gamma^{r-1} \psi(0)} + \frac{1}{\gamma^{r-1}}
\end{align*}
for all $t\ge 0$, where we have use that $\|e(0)\|\le \gamma^r \psi(0)$. Hence we obtain that
\begin{align*}
     \tfrac{1}{2}\dd{t}\Norm{\frac{e}{\psi}}^2 \!\!\!
     &\le\! -\tfrac12 (k_1 - \alpha) (1+\gamma) + \gamma k_1 + \frac{\Norm{\dot e_1^0}}{\gamma^{r-1} \psi(0)} + \frac{1}{\gamma^{r-1}}\\
     &\le\! -\tfrac12 (1-\gamma) k_1 + \alpha + \frac{\Norm{\dot e_1^0}}{\gamma^{r-1} \psi(0)} + \frac{1}{\gamma^{r-1}} \le 0
\end{align*}
for all $t \in [t_{\star},t^\star]$, where the last inequality follows from~\eqref{eq:cond-k_i}. Now consider the case $i>1$. Then we have
$-\tfrac{\dot{\psi}_i(t)}{\psi_i(t)}\le  \alpha$ for all $t\ge 0$ and, invoking that by~\eqref{eq:ErrorVarDyn}
\[
    \|e_i^0\| \le \Norm{\dot e_{i-1}^0} + k_{i-1} \Norm{e_{i-1}^0},
\]
we find that
\begin{align*}
    \frac{\psi_{i+1}(t)}{\psi_i(t)}\!&=\!  
    \frac{\frac{1}{\gamma^{r-i}}\big(\!\Norm{\dot e_i^0}\!\! +\! k_i\! \Norm{e_i^0}\!\big) \me^{-\alpha t} \!+\! \frac{\beta}{\alpha\gamma^{r-1}}}
    {\frac{1}{\gamma^{r-i+1}}\big(\!\Norm{\dot e_{i-1}^0}\!\! +\! k_{i-1}\! \Norm{e_{i-1}^0}\!\big) \me^{-\alpha t} \!+\! \frac{\beta}{\alpha\gamma^{r-1}}}\\
    &\le\! \gamma \frac{\Norm{\dot e_i^0} + k_i \Norm{e_i^0}}{\Norm{\dot e_{i-1}^0} + k_{i-1} \Norm{e_{i-1}^0} + \frac{\beta}{\alpha \gamma^{i-2}}} + 1\\
    &\le\! \gamma k_i + \gamma \frac{\Norm{\dot e_i^0}}{\Norm{e_i^0}+ \frac{\beta}{\alpha \gamma^{i-2}}} + 1
\end{align*}
for all $t\ge 0$. Hence we obtain that
\begin{align*}
     \tfrac{1}{2}\dd{t}\Norm{\frac{e_i}{\psi_i}}^2 
     \!\!\!&\le\! -\tfrac12 (k_i \!-\! \alpha) (1+\gamma) \!+\! \gamma k_i \!+\! \gamma \frac{\Norm{\dot e_i^0}}{\Norm{e_i^0}\!+\! \frac{\beta}{\alpha \gamma^{i-2}}} + 1\\
     &\le\! -\tfrac12 (1-\gamma) k_i + \alpha + \gamma \frac{\Norm{\dot e_i^0}}{\Norm{e_i^0}+ \frac{\beta}{\alpha \gamma^{i-2}}} + 1 \le 0
\end{align*}
for all $t \in [t_{\star},t^\star]$, where the last inequality follows from~\eqref{eq:cond-k_i}. Summarizing, in each case the contradiction 
\[
    1 \leq \| e_i(t^*)/\psi_i(t^*)\|^2 \leq \| e_i(t_*)/\psi_i(t_*)\|^2 = \e^2<1
\]
arises, which completes the proof.
\end{proof}

For $\hat{t}\geq 0$, $M>0$, $T>0$ and $\zeta\in\cY_{\hat{t}}^r$,
let $\hat{y}:=\zeta|_{I_{\sigma}^{\hat{t},\tau}}$ and $\hat{\oT}:=\oT(\chi(\zeta))|_{I_{0}^{\hat{t},\tau}}$,
where $I_{\sigma}^{\hat{t},\tau}:=[-\sigma,\hat{t}]\cap[\hat{t}-\tau,\hat{t}]$.
We denote by $\cU_{T}(M,\hat{t},\hat{y},\hat{\oT})$ the set
\begin{align}\label{eq:Def-U}
        \setdef{\!\!
        u\!\in\! L^{\infty}([\hat{t},\hat{t}\!+\!T],\R^m)
        \!\!\!}
        {\!\!\!\!
        \begin{array}{l}
            x(t;\hat{t},\hat{y},\hat{\oT},u) - \chi(y_{\rf})(t)\in\cD_{t}^r\\
            \text{for all }t\in [\hat{t},\hat{t}\!+\!T],\SNorm{u}\le M
        \end{array} 
        \!\!\!\!\!}\!.
\end{align}
 This is the set of all $L^\infty$-controls $u$ bounded by $M$ which, if applied to the system~\eqref{eq:Sys},
guarantee that the error signals $e_{i}(x(t;\hat{t},\hat{y},\hat{\oT},u) - \chi(y_{\rf})(t))$ evolve within their respective 
funnels defined by~$\psi_i$ on the interval $[\hat{t},\hat{t}+T]$. 
We note that the conditions in~\eqref{eq:Def-U} implicitly require the solution $x(\cdot;\hat{t},\hat{y},\hat{\oT},u)$ to exist on the interval $[\hat{t},\hat{t}+T]$.

\begin{lemma}\label{Th:BoundM}
    Under the assumptions of Theorem~\ref{Th:FunnelMPC},
    consider the functions $\psi_2,\ldots,\psi_r\in\cG$ defined in~\eqref{eq:psi_i}.
    For $\hat{t}\geq 0$, $\zeta\in\cY_{\hat{t}}^r$,
    let $\hat{y}:=\zeta|_{I_{\sigma}^{\hat{t},\tau}}$ and $\hat{\oT}:=\oT(\chi(\zeta))|_{I_{0}^{\hat{t},\tau}}$.
    Then, there exists $M>0$ such that for all $T>0$ we have
    \[
        \cU_{T}(M,\hat{t},\hat{y},\hat{\oT})\neq\emptyset.
    \]
    Furthermore, 
    \begin{multline}\label{eq:RecursiveControl}
        \fa T_1,T_2>0 \fa u\in \cU_{T_1}(M,\hat{t},\hat{y},\hat{\oT})
        \fa t\in [\hat{t},\hat{t}+T_1]:\\
        \cU_{T_2}(M,t,y(\cdot;\hat{t},\hat{y},\hat{\oT},u)|_{I_{\sigma}^{t,\tau}},\oT(\chi(\tilde{\zeta}))|_{I_{0}^{t,\tau}} )\neq\emptyset,
    \end{multline}
    where $\tilde{\zeta}$ is an arbitrary element in $\cY_t^r$ with $\tilde{\zeta}|_{I_{\sigma}^{t,\tau}}=y(\cdot;\hat{t},\hat{y},\hat{\oT},u)|_{I_{\sigma}^{t,\tau}}$.
\end{lemma}

\begin{proof}
    \emph{Step 1}: We define $M>0$. 
    To that end, define, for $i=1,\ldots, r-1$ and $j=0,\ldots,r-i-1$
    \begin{equation}\label{eq:DefMus}
        \mu_i^0 := \SNorm{\psi_i},\quad \mu_{i}^{j+1}:= \mu_{i+1}^{j}+k_{i}\mu_{i}^{j}.
    \end{equation}
    Using the constants $f_{\max}$ and $g_{\max}$ from \Cref{Lemma:DynamicBounded}, define
    \[
        M:=g_{\max}
        \rbl
            f_{\max}+\SNorm{y_{\rf}^{(r)}}+\sum_{j=1}^{r-1}k_j\mu_{j}^{r-j}+\SNorm{\dot{\psi}_r}
        \rbr.
    \]

    \noindent
    \emph{Step 2}: Let $T>0$ be arbitrary. We construct a control function $u$ and show that $u\in\cU_{T}(M,\hat{t},\hat{y},\hat{\oT})$.
    To this end, for some $u\in L^{\infty}([\hat{t},\hat{t}+T],\R^m)$, we use the shorthand notation $y(t):= y(t;\hat{t},\hat{y},\hat{\oT},u)$, $e(t):= y(t) - y_{\rf}(t)$ and
    $e_i(t):= e_{i}(\chi(y-y_{\rf})(t))$ for $i=1,\ldots,r$.
    The application of the output feedback
    \begin{multline}\label{eq:DefFeedbackControl}
        u(t):=g(\oT(\chi(y))(t))^{-1}
       \rbl 
            -f(\oT(\chi(y))(t))+y^{(r)}_{\rf}(t)\right. \\
            \left.-\sum_{j=1}^{r-1} k_{j}e^{(r-j)}_j(t)+e_r(t)\tfrac{\dot{\psi}_r(t)}{\psi_r(t)}
        \rbr
    \end{multline}
    to the system~\eqref{eq:Sys} leads to a closed-loop system.
    If this initial value problem is considered on the interval~$[\hat{t},\hat{t}+T]$ with initial conditions
    \begin{align*}
      x|_{I_{\sigma}^{\hat{t},\tau}} &= \chi(\hat y),\\
       \oT(x)|_{I_{0}^{\hat{t},\tau}} &= \hat \oT,
    \end{align*}
    then an application of a variant (a straightforward modification tailored to the current context) of~\cite[Thm. B.1]{IlchRyan09}
    yields the existence of a maximal solution~$x:[-\sigma,\omega)\to\R^{rm}$ since 
    $\hat{y}=\zeta|_{I_{\sigma}^{\hat{t},\tau}}$ and $\hat{\oT}=\oT(\chi(\zeta))|_{I_{0}^{\hat{t},\tau}}$ for some $\zeta\in\cY_{\hat{t}}^r$ by assumption.
    If~$x$ is bounded, then $\omega=\infty$, so the solution exists on $[-\sigma, \hat t + T]$.
    Utilizing~\eqref{eq:ErrorVarDyn} one can show that
    \begin{equation}\label{eq:e_r-e_j}
        e_r(t)=e^{(r-1)}(t)+\sum_{j=1}^{r-1}k_je_j^{(r-j-1)}(t).
    \end{equation}
    Omitting the dependency on $t$, we calculate for $t \in [\hat{t},\omega)$:
    \begin{align*}
    &\frac{\dot{e}_r\psi_r-e_r\dot{\psi}_r}{\psi_r}
    =e^{(r)}+\sum_{j=1}^{r-1}k_je_j^{(r-j)}-e_r\frac{\dot{\psi}_r}{\psi_r}\\
    &=\!\!f(\oT(\chi(y)))\!+\!g(\oT(\chi(y)))u\!-\!y_{\rf}^{(r)}\!+\!\sum_{j=1}^{r-1}\!k_je_j^{(r-j)}\!\!-\!e_r\frac{\dot{\psi}_r}{\psi_r}
    \!=\!0.
    \end{align*}
    Therefore,
    \[
    \dd{t}\tfrac{1}{2}\Norm{\frac{e_r}{\psi_r}}^2
    = \al\frac{e_r}{\psi_r},\frac{\dot{e}_r\psi_r-e_r\dot{\psi}_r}{\psi_r^2}\ar
    =0.
    \]
   Since $\Norm{\tfrac{e_r(\hat{t})}{\psi_r(\hat{t})}}<1$ by the assumption $\chi(\hat y - y_{\rf})(\hat{t})\in\cD_{\hat{t}}^r$, this yields $\Norm{\tfrac{e_r(t)}{\psi_r(t)}}<1$ for all $t\in[\hat{t},\omega)$.
   This implies, according to \Cref{Lemma:Errors},  $\chi(y-y_{\rf})(t)\in\cD_{t}^r$ for all $t\in[\hat{t},\omega)$, 
   i.e., $\Norm{e_i(t)}<\psi_i(t)$ for all $i=1,\ldots, r$. Thus, $\Norm{e_i(t)}\leq\mu_i^0$ for all $i=1,\ldots, r$.
   Invoking boundedness of $y_{\rf}^{(i)}$, $i=0,\ldots,r$, and the relation in~\eqref{eq:reps_ei_chi}, we may infer that $x = \chi(y)$ is bounded on $[\hat{t},\omega)$.
   Hence, $\omega=\infty$. 
   Note that, on the interval $[\hat{t},\hat{t}+T]$, the function $\oT(x)(\cdot)$ is solely determined by $\hat{y}$, $\hat{\oT}$, and the differential equation~\eqref{eq:Sys}.
   It does not depend on the values of $x(t)$ prior to $\max\{-\sigma,\hat{t}-\tau\}$ due to the limited memory property of the operator~$\oT$.
   Since  $\hat{y}=\zeta|_{I_{\sigma}^{\hat{t},\tau}}$ and $\hat{\oT}=\oT(\chi(\zeta))|_{I_{0}^{\hat{t},\tau}}$ for some $\zeta\in\cY_{\hat{t}}^r$,
   we thus have $\Norm{f(\oT(\chi(y))(t))}\leq f_{\max}$ and 
   $\Norm{g(\oT(\chi(y))(t))^{-1}}\leq g_{\max}$ for all $t\in[\hat{t},\hat{t}+T]$ according to~\Cref{Lemma:DynamicBounded}.
   Finally, using~\eqref{eq:ErrorVarDyn} and the definition of~$\mu_i^j$ it follows that
   \begin{equation}\label{eq:EiDotBounded}
        \Norm{e^{(j+1)}_i(t)}
        \!\!=\!\!
        \Norm{e^{(j)}_{i+1}(t)\!-\!k_ie^{(j)}_{i}(t)}
        \!\leq\!
        \mu_{i+1}^j\!+k_i\mu^{j}_{i}
        \!=\!\mu^{j+1}_{i}\!
   \end{equation}
   inductively for all $i=1,\ldots, r-1$ and $j=0,\ldots,r-i-1$. Thus, by definition of~$u$ and~$M$ we have $\SNorm{u}\leq M$ and hence $u\in\cU_{T}(M,\hat{t},\hat{y},\hat{\oT})$.
   
   \noindent
   \emph{Step 3}: We show implication~\eqref{eq:RecursiveControl}.
   If, for any $T_1>0$ an arbitrary but fixed control $\hat{u}\in\cU_{T_1}(M,\hat{t},\hat{y},\hat{\oT})$ is applied to the system~\eqref{eq:Sys},
   then $x(t;\hat t,\hat y, \hat{\oT}, u) - \chi(y_{\rf})(t)\in\cD_t^r$ for all $t\in[\hat{t},\hat{t}+T_1]$.
   If for any $\tilde{t}\in[\hat{t},\hat{t}+T]$, the system is considered on the interval 
   $[\tilde{t},\tilde{t}+T_2]$ with $T_2>0$ and initial data $\tilde y := y(\cdot;\hat{t},\hat{y}, \hat{\oT},{u})|_{I_{\sigma}^{\tilde{t},\tau}}$
   and $\tilde{\oT}=\oT(\chi(\tilde{\zeta}))|_{I_{0}^{\tilde{t},\tau}}$ for some $\cY_t^r$ with $\tilde{\zeta}|_{I_{\sigma}^{\tilde{t},\tau}}=y(\cdot;\hat{t},\hat{y},\hat{\oT},u)|_{I_{\sigma}^{\tilde{t},\tau}}$,
   then one can show by a repetition of the arguments in Step~2 that the application 
   of the control~$\tilde u\in L^\infty([\tilde{t},\tilde{t}+T_2],\R^m)$ as in~\eqref{eq:DefFeedbackControl}, \textit{mutatis mutandis}, guarantees 
   $x(t;\tilde t,\tilde  y, \tilde{\oT}, \tilde u) - \chi(y_{\rf})(t)\in\cD_t^r$ for all $t\in[\tilde{t},\tilde{t}+T_2]$.
  Note that, on the interval $[\tilde{t},\tilde{t}+T]$, the function $\oT(x)(\tilde{t})$ is solely determined by $\tilde{y}$, $\tilde{\oT}$, and the differential equation~\eqref{eq:Sys}. 
   It does not depend on the choice of $\tilde{\zeta}$ due to the limited memory property of the operator~$\oT$.
   Since the prerequisites for~\Cref{Lemma:DynamicBounded,Lemma:Errors} are still satisfied, the control $\tilde u$ is bounded by $M$ as constructed in Step~1.
   Thus, $\tilde u\in\cU_{T_2}(M,\tilde t,\tilde y, \tilde{\oT})\neq\emptyset$.
\end{proof}

\begin{lemma} \label{Th:FiniteJImplFunnel}
    Under the assumptions of Theorem~\ref{Th:FunnelMPC},
    consider the functions $\psi_2,\ldots,\psi_r\in\cG$ defined in~\eqref{eq:psi_i}.
    For $\hat{t}\ge 0$, $\zeta\in\cY_{\hat{t}}^r$,
    let $\hat{y}:=\zeta|_{I_{\sigma}^{\hat{t},\tau}}$ and $\hat{\oT}:=\oT(\chi(\zeta))|_{I_{0}^{\hat{t},\tau}}$.
    Further, let $T>0$ $M>0$ such that $\cU_{T}(M,\hat{t},\hat{y}, \hat{\oT})\neq\emptyset$.
    Then, $\cU_{T}(M,\hat{t},\hat{y}, \hat{\oT})$ is equal to the set $\tilde\cU_{T}(M,\hat{t},\hat{y}, \hat{\oT})$ defined by
    \[
    \begin{small}
        \setdef
            {\!\!\!
            u\in L^{\infty}([\hat{t},\hat{t}\!+\!T],\R^m)
            \!\!\!}
            {\!\!\!\!
                 \begin{array}{l}
                x(t;\hat{t},\hat{y}, \hat{\oT},u)\text{ satisfies \eqref{eq:Sys} for all }\\ t\in[\hat{t},\hat{t}\!+\!T],\SNorm{u} \le M,\\
                \int_{\hat{t}}^{\hat{t}+T} \ell_{\psi_r}(t,\zeta(t),u(t))\, {\rm d} t < \infty,\\
                \zeta(t):=x(t; \hat{t}, \hat{y}, \hat{\oT} ,u)-\chi(y_{\rf})(t)
            \end{array}\!\!\!\!\!\!\!
            }\!.
    \end{small}
    \]
\end{lemma}
\begin{proof}
    We adapt the proof of~\cite[Thm.~4.3]{BergDenn21} to the current setting.
    Given $u\in \cU_{T}(M,\hat{t},\hat{y}, \hat{\oT})$, it follows from the definition of $\cU_{T}(M,\hat{t},\hat{y}, \hat{\oT})$ that
    $\zeta(t):=x(t; \hat{t}, \hat{y}, \hat{\oT} ,u)-\chi(y_{\rf})(t)\in\cD_t^r$ for all $t\in[\hat{t},\hat{t}+T]$.
    Thus,
    \[
        \forall\, t\in[\hat{t},\hat{t}+T]:\ \Norm{e_r(\zeta(t))}<\psi_r(t).
    \]
    We use the shorthand notation 
    $e_r(t):=e_r(\zeta(t))$.
    Due to continuity of  the involved functions, there exists $\e\in(0,1)$ with
    $\Norm{e_r(t)}^2 \le \psi_r(t)^2 - \e$ for all $t \in [\hat{t},\hat{t}+T]$.
    Then, $\ell_{\psi_r}(t,\zeta(t),u(t)) \ge 0$ for all $t \in [\hat{t},\hat{t}+T]$ and
    \begin{align*}
        &       \!\int_{\hat{t}}^{\hat{t}+T}\!        \Abs{ \ell_{\psi_r}(t,\zeta(t),u(t))}\d{t}\\
        &=      \!\int_{\hat{t}}^{\hat{t}+T}\!\Abs{ \frac{\Norm{e_r(t)}^2}{\psi_r(t)^2-\Norm{e_r(t)}^2} + \l_u\Norm{u(t)}^2}\d{t}\\
        &\leq   \!\int_{\hat{t}}^{\hat{t}+T}\!{ \frac{\|\psi_r\|_\infty}{\e} + \l_u\! \SNorm{u}^2}\!\d{t} 
        \le\! \left(\!\frac{\|\psi_r\|_\infty}{\e} + \l_u M^2\!\right)\! T  <  \!    \infty.
    \end{align*}
    Therefore,  $\cU_{T}(M,\hat{t},\hat{y},\hat \oT)$ is contained in $\tilde{\cU}_{T}(M,\hat{t},\hat{y},\hat \oT)$.

    Let $u\in\tilde{\cU}_{T}(M,\hat{t},\hat{y},\hat \oT)$. We show that $\zeta(t)\in\cD_t^r$ for all $t\in[\hat{t},\hat{t}+T]$.
    Since $\cU_{T}(M,\hat{t},\hat{y},\hat \oT)\neq\emptyset$, we have $\chi(\zeta)(\hat t) = \chi(\hat{y}-y_{\rf})(\hat{t})\in\cD_{\hat{t}}^r$.
    According to~\Cref{Lemma:Errors} it suffices to show that $\Norm{e_r(t)}< \psi_r(t)$ for all $t\in[\hat{t},\hat{t}+T]$.
    Assume there exists $t\in[\hat{t},\hat{t}+T]$ with $\Norm{e_r(t)}> \psi_r(t)$. 
    By continuity of the involved functions, there exists 
    \[
        \tilde{t}:=\min\setdef
        {t\in[\hat{t},\hat{t}+T]}
        {\Norm{e_r(t)}=\psi_r(t)}.
    \]
    Recalling the definition of the Lebesgue integral, see e.g.~\cite[Def 11.22]{Rudi76},
    $\int_{\hat{t}}^{\hat{t} + T}\ell_{\psi_r}(t,\zeta(t),u(t))\, {\rm d} t < \infty$ implies
    $\int_{\hat{t}}^{\hat{t} + T}\rbl \ell_{\psi_r}(t,\zeta(t),u(t))\rbr^+ \d t < \infty$, where $\rbl\ell_{\psi_r}(t,\zeta(t),u(t))\rbr^+:=\max\{ \ell_{\psi_r}(t,\zeta(t),u(t)),0\}$.
    Thus,
    \begin{align*}
       &\int_{\hat{t}}^{\tilde t}\!\!    \frac{1}{1-\frac{\Norm{e_r(t)}^2}{\psi_r(t)^2}}            \!\d t = \int_{\hat{t}}^{\tilde t}\!\!    \frac {\Norm{e_r(t)}^2}{\psi_r(t)^2  - \Norm{e_r(t)}^2} + 1           \!\d t \\
       & \leq  \int_{\hat{t}}^{\hat{t}+T}\!\!\!\rbl\frac {\Norm{e_r(t)}^2}{\psi_r(t)^2  - \Norm{e_r(t)}^2}\rbr^{\!\!+}    \d t + T\\
       &\leq   \int_{\hat{t}}^{\hat{t}+T}\!\!\!\rbl\frac {\Norm{e_r(t)}^2}{\psi_r(t)^2  - \Norm{e_r(t)}^2}\rbr^{+}  +\lambda _u\Norm{u}^2\d t + T\\
       &=     \int_{\hat{t}}^{\hat{t}+T}\!\!\!\rbl \ell_{\psi_r}(t,\zeta(t),u(t))\rbr^{+} \!\!\d t + T < \infty.
    \end{align*}
    We seek to apply~\cite[Lem.~4.1]{BergDenn21}.
    To this end, we show that $1-\frac{\Norm{e_r(\cdot)}^2}{\psi_r(\cdot)^2}$ is Lipschitz continuous on $[\hat t, \tilde t]$.
    If a function is bounded, with bounded derivative, then it is Lipschitz continuous.
    Hence, since $\psi_r\in\cG$ it is Lipschitz continuous on $[\hat t, \tilde t]$.
    Clearly, $e_r$ is bounded by $\psi_r$, and we show that $\dot e_r$ is bounded.
    First observe that, since $x(\cdot; \hat{t}, \hat{y},\hat \oT ,u)$ is continuous,
    it is bounded on the compact interval $[\hat{t},\tilde{t}]$.
    Since~$f$ and~$g$ are continuous and by the BIBO property of the operator~$\oT$, $f(\oT(x(\cdot; \hat{t}, \hat{y},\hat \oT ,u)))$ and $g(\oT(x(\cdot; \hat{t}, \hat{y},\hat \oT ,u)))$ are bounded 
    on the interval~$[\hat{t},\tilde{t}]$. As in \eqref{eq:EiDotBounded} in the proof proof of \Cref{Th:BoundM},
    for $e_i(t):= e_i(\zeta)(t)$, we may show that $e_{i}^{(j)}$ is bounded by $\mu_i^j$ as defined in~\eqref{eq:DefMus} for $i=1,\ldots, r$ and $j=0,\ldots r-i-1$.
    Finally, it follows from~\eqref{eq:e_r-e_j} that
    \[
        \dot{e}_r=f(\oT(\chi(y))(\cdot))+g(\oT(\chi(y))(\cdot))u-y_{\rf}^{(r)}
        +\sum_{j=1}^{r-1}k_je_j^{(r-j)},
    \]
    which is bounded on~$[\hat{t},\tilde{t}]$ by the above observations. Since $\psi_r$ and $e_r$ are Lipschitz continuous and products and sums of Lipschitz continuous functions on a compact interval are again Lipschitz continuous, we may infer that $1-\frac{\Norm{e_r(\cdot)}^2}{\psi_r(\cdot)^2}$ is  Lipschitz continuous on $[\hat t, \tilde t]$.
    Now~\cite[Lem.~4.1]{BergDenn21} yields that it is strictly positive, i.e., $\psi_r(t)^2  > \Norm{e_r(t)}^2$ for all $t\in[\hat t, \tilde t]$,
    which contradicts the definition of~$\tilde{t}$.
    Hence,  $\tilde{\cU}_{T}(M,\hat{t},\hat{y},\hat \oT)\subseteq \cU_{T}(M,\hat{t},\hat{y},\hat \oT)$.
\end{proof}

\begin{lemma}\label{Th:ExistenceSolution}
    Under the assumptions of Theorem~\ref{Th:FunnelMPC},
    consider the functions $\psi_2,\ldots,\psi_r\in\cG$ defined in~\eqref{eq:psi_i}.
    For $\hat{t}\ge 0$, $\zeta\in\cY_{\hat{t}}^r$,
    let $\hat{y}:=\zeta|_{I_{\sigma}^{\hat{t},\tau}}$ and $\hat{\oT}:=\oT(\chi(\zeta))|_{I_{0}^{\hat{t},\tau}}$.
    Further, let $T>0$, $M>0$, such that $\cU_{T}(M,\hat{t},\hat{y}, \hat{\oT})\neq \emptyset$.
    Then, there exists $u^\star\in\cU_{T}(M,\hat{t},\hat{y}, \hat{\oT})$ such that $u^\star$ solves the
    OCP~\eqref{eq:FMPC_OCP} for~$\theta = \psi_r$.
\end{lemma}
\begin{proof}
    As a consequence of~\Cref{Th:FiniteJImplFunnel}, solving the OCP~\eqref{eq:FMPC_OCP} is  equivalent to minimizing the function
    \begin{align*}
        J&:L^\infty([\hat{t},\hat{t}+T],\R^m)\to\R\cup\cbl\infty\cbr,\\
    J(u)&:=
    \begin{dcases}
        \int_{\hat{t}}^{\hat{t}+T}\ell_{\psi_r}(t,\zeta(t),u(t))\!\d t,  &\!\! u\in \cU_{T}(M,\hat{t},\hat{y}, \hat{\oT}),\\
        \infty,&\!\!\text{else},
    \end{dcases}
    \end{align*}
    where $\zeta(t):=x(t; \hat{t}, \hat{y}, \hat{\oT} ,u)-\chi(y_{\rf})(t)$.
    For every $u\in \cU_{T}(M,\hat{t},\hat{y}, \hat{\oT})$ we have $\Norm{e_r(\zeta(t))}<\psi_r(t)$    for all $t\in[\hat{t},\hat{t}+T]$, thus $J(u)\geq 0$.
    Hence, the infimum $J^\star:=\inf_{u\in\cU_{T}(M,\hat{t},\hat{y}, \hat{\oT})}J(u)$ exists.
    Let $(u_k)\in\rbl\cU_{T}(M,\hat{t},\hat{y}, \hat{\oT})\rbr^\N$ be a minimizing sequence, meaning $J(u_k)\to J^\star$.
    By definition of~$\cU_{T}(M,\hat{t},\hat{y}, \hat{\oT})$, we have~$\Norm{u_k}\leq M$ for all $k\in\N$.
    Since $L^{\infty}([\hat{t},\hat{t}+T],\R^m)\subset L^{2}([\hat{t},\hat{t}+T],\R^m)$, we conclude that $(u_k)$ 
    is a bounded sequence in the Hilbert space $L^2$, thus $u_k$ converges weakly,
    up to a subsequence, to a function $u^\star\in L^{2}([\hat{t},\hat{t}+T],\R^m)$.
    Let $(x_k):=(x(\cdot;\hat{t},\hat{y}, \hat{\oT},u_k)|_{I_{\sigma}^{\hat{t}+T,\tau}})\in\con(I_{\sigma}^{\hat{t}+T,\tau},\R^{rm})^\N$ be 
    the sequence of associated responses.
    By $u_k\in\cU_{T}(M,\hat{t},\hat{y}, \hat{\oT})$ we have $x_k(t)\in\cD_{t}^r$ for all $t$ in $[\hat{t},\hat{t}+T]$.
    Since the set $\bigcup_{t\in[\hat{t},\hat{t}+T]}\cD_{t}^r$ is compact and independent of $k\in\N$,
    the sequence $(x_k)$ is uniformly bounded.
    A repetition of Steps 2--4 of the proof of~\cite[Thm.~4.6]{BergDenn21} yields that $(x_k)$ has a subsequence (which we do not relabel)
    that converges uniformly to $x^\star=x(\cdot;\hat{t},\hat{y}, \hat{\oT},u^\star)$ and that $\SNorm{u^\star}\leq M$.
    Along the lines of Steps 5--7 of the proof of~\cite[Thm.~4.6]{BergDenn21} it follows that
    $u^\star\in\cU_{T}(M,\hat{t},\hat{y}, \hat{\oT})$ and $J(u^\star)=J^\star$. This completes the proof.
\end{proof}

\subsection{Proof of~\Cref{Th:FunnelMPC}}
    Choosing the bound $M>0$ from~\Cref{Th:BoundM} and utilizing~\Cref{Th:ExistenceSolution}
    this can be shown by a straightforward adaption of the proof  of~\cite[Thm.~2.10]{BergDenn21}. \hfill$\Box$

\section{Simulations}\label{Sec:Sim}
To demonstrate the application of the funnel MPC~\Cref{Algo:FMPC} we consider the example of 
a mass-spring system mounted on a car from~\cite{SeifBlaj13}.
Consider a car with mass~$m_1$, on which a ramp is mounted and inclined by the angle~$\theta\in[0,\tfrac{\pi}{2})$.
On this ramp a mass~$m_2$, which is coupled to the car by spring-damper component 
with spring constant $k>0$ and damping coefficient $d>0$, moves frictionless, see~\Cref{Fig:Mass_on_car}.
\begin{figure}[htp]
    \begin{center}
    \includegraphics[trim=2cm 4cm 5cm 15cm,clip=true,width=6.5cm]{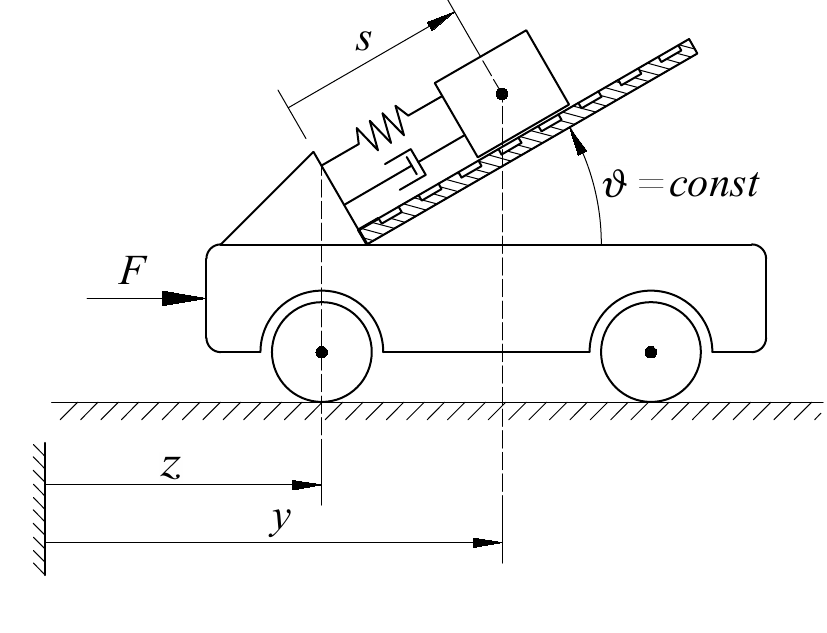}
    \end{center}
    \vspace*{-8mm}
    \caption{Mass-on-car system.}
    \vspace*{-2mm}
    \label{Fig:Mass_on_car}
\end{figure}
A control force $F=u$ can be applied to the car. 
The dynamics of the system can be described by the equations
\begin{align}\label{eq:ExampleMassOnCarSystem}
\begin{small}
    \begin{bmatrix}
        m_1 + m_2& m_2\cos(\vartheta)\\
        m_2 \cos(\vartheta) & m_2
    \end{bmatrix}
    \begin{pmatrix}
        \ddot{z}(t)\\
        \ddot{s}(t)
    \end{pmatrix}
    \!+\!
    \begin{pmatrix}
        0\\
        k s(t) +d\dot{s}(t)
    \end{pmatrix}
    \!=\!
    \begin{pmatrix}
        u(t)\\
        0
    \end{pmatrix},
\end{small}
\end{align}
where $z(t)$ is the horizontal position of the car and $s(t)$  the relative position of the mass on
the ramp at time $t$. The output~$y$ of the system is the horizontal position of the mass on the ramp, given by
\[
    y(t) = z(t) +s(t)\cos(\vartheta).
\]
For the simulation we choose the parameters $m_1=4$, $m_2=1$, $k=2$, $d=1$, $\vartheta=\tfrac{\pi}{4}$, and initial values $z(0)=s(0)=\dot{z}(0)=\dot{s}(0)=0$.
As outlined in~\cite{BergIlch21}, for these parameters the system~\eqref{eq:ExampleMassOnCarSystem} can be written in the form~\eqref{eq:Sys}  with $\sigma=0$,
where $f$ is the identity function on $\R$, $g$ is a non-zero constant function, and the operator $\oT$ has the form $\oT:\con(\Rp,\R^2)\to L_{\loc}^\infty(\Rp,R)$, 
$\oT(y,\dot{y})(t):=R_1y(t)+R_2\dot{y}(t)+S(\me^{Qt}\eta_0+\int_{0}^t\me^{Q(t-s)}Py(s)\d{s})$
for suitable real numbers $R_1$, $R_2$, vectors $S^\top$, $P$, $\eta_0\in\R^2$, and a matrix $Q \in \R^{2\times2}$.
As detailed in~\cite{BergIlch21}, the operator $\oT$ satisfies the causality, local Lipschitz, and bounded-input bounded-output property from Definition~\ref{Def:OperatorClass}.
To see that it also satisfies the limited memory property for $\tau=0$, let $\hat{t}\geq0$  and $(y_1,\dot{y}_1),(y_2,\dot{y}_2)\in\con(\Rp,\R^2)$ 
with $\oT(y_1,\dot{y}_1)(\hat{t})=\oT(y_2,\dot{y}_2)(\hat{t})$ and $(y_1,\dot{y}_1)(t)=(y_2,\dot{y}_2)(t)$ for all $t\geq\hat{t}$.
Using the shorthand notation $V(y,\dot{y})(t):=R_1y(t)+R_2\dot{y}(t)+S\me^{Qt}\eta_0$ and $L(y,\dot{y})(t):=S\int_{0}^t\me^{Q(t-s)}Py(s)\d{s}$, 
we have $\oT(y,\dot{y})(t)=V(y,\dot{y})(t)+L(y,\dot{y})(t)$. It is clear that $V(y_1,\dot{y}_1)(t)=V(y_2,\dot{y}_2)(t)$ for all $t\geq\hat{t}$.
To show the limited memory property for the operator~$\oT$, it therefore is sufficient to show $L(y_1,\dot{y}_1)(t)=L(y_2,\dot{y}_2)(t)$ for all $t\geq\hat{t}$.
For $t\geq \hat{t}$, we have
\begin{align*}
     &L(y_1,\dot{y}_1)(t)=S\int_{0}^t\me^{Q(t-s)}Py_1(s)\d{s}\\
                        &=L(y_1,\dot{y}_1)(\hat t)+S\int_{\hat{t}}^t\me^{Q(t-s)}Py_1(s)\d{s}\\
                        &=\oT(y_1,\dot{y}_1)(\hat{t})-V(y_1,\dot{y}_1)(\hat{t})+S\int_{\hat{t}}^t\me^{Q(t-s)}Py_1(s)\d{s}\\
                        &=\oT(y_2,\dot{y}_2)(\hat{t})-V(y_2,\dot{y}_2)(\hat{t})+S\int_{\hat{t}}^t\me^{Q(t-s)}Py_2(s)\d{s}\\
                        &=L(y_2,\dot{y}_2)(\hat t)+S\int_{\hat{t}}^t\me^{Q(t-s)}Py_2(s)\d{s}\\
                        &=S\int_{0}^t\me^{Q(t-s)}Py_2(s)\d{s}=L(y_2,\dot{y}_2)(t).
\end{align*}%
Thus, the operator~$\oT$ belongs to the class $\cT^{2,1}_{0}$ and hence system~\eqref{eq:ExampleMassOnCarSystem} belongs to the class~$\cN^{1,2}$ for the given parameters.
The control objective is tracking of the reference signal $y_{\rf}(t) = \cos(t)$ within predefined boundaries described by a function~$\psi\in\cG$.
This means that the tracking error $e(t) = y(t)-y_{\rf}(t)$ should satisfy $\Norm{e(t)}< \psi(t)$ for all $t\geq 0$.

We compare the funnel MPC Algorithm~\Cref{Algo:FMPC} with the original funnel MPC scheme from~\cite{BergDenn21},
for which only feasibility for systems with relative degree one has been shown so far, and the funnel MPC scheme from~\cite{BergDenn22}, 
which uses feasibility constraints to ensure recursive feasibility for systems with higher relative degree.
Since, in comparison to the set~$\cG$, the set of admissible funnel functions for control scheme from~\cite{BergDenn22} is quite restrictive, 
the same funnel function~$\psi(t) =1/10 + 11\me^{-27x/20} - 7\me^{-3x/2} $ as in~\cite{BergDenn22} was chosen.
Straightforward calculations show that  $\alpha=1.5$, $\beta=\tfrac{3}{20}$, $\gamma =0.5$, and $k_1=14$ satisfy the requirements of~\Cref{Th:FunnelMPC}.
With these parameters, the funnel function $\theta$ in~\eqref{eq:DefTheta} is given by
\[
    \theta(t) = 28 \me^{-3t/2} + \frac{1}{5}.
\]
For the stage cost function~$\ell_\theta$ as in~\eqref{eq:stageCostFunnelMPC} the parameter~$\lambda_u=\tfrac{1}{100}$ has been chosen.
Further, the maximal input was limited to $\SNorm{u}\leq 20$, i.e., $M=20$ was chosen.
Inserting the definition of the function~$e_2$ from~\eqref{eq:ErrorVar} with parameter~$k_1$, the stage cost thus reads
\[
    \ell_\theta(t,\xi_1,\xi_2,u) = \frac{\Norm{\xi_2+k_1 \xi_1}^2}{\theta(t)^2 
        - \Norm{\xi_2+k_1 \xi_1}^2} + \l_u \Norm{u}^2
\]
for $\Norm{\xi_2+k_1 \xi_1}\neq \theta(t)$. With this and $e(t) = y(t) - y_{\rf}(t)$ the OCP~\eqref{eq:FMPC_OCP} becomes
\[
            \mathop
            {\operatorname{minimize}}_{\mathcal{\substack{u\in L^{\infty}([\hat{t},\hat{t}+T],\R^{m}\!),\\\SNorm{u}\leq M}}} 
            \int_{\hat{t}}^{\hat{t}+T}\!\!\!\ell_\theta(t,e(t),\dot e(t),u(t))\!\d t .
\]
As in~\cite{BergDenn22} and~\cite{BergDenn21}  only step functions with constant step length $0.04$ are considered in the OCP~\eqref{eq:FMPC_OCP} due to discretisation.
The prediction horizon and the time shift are chosen as~$T=0.6$ and $\delta=0.04$.

\begin{figure}[ht] \centering
    \begin{subfigure}[b]{0.48\textwidth}
     \centering
        \includegraphics[width=8.8cm]{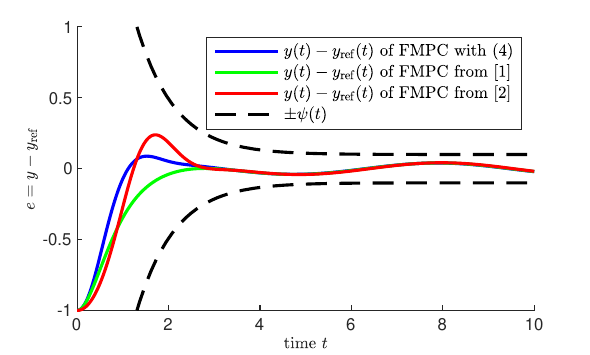}
        \caption{\vspace{-2mm}Tracking error~$e$ and funnel boundary~$\psi$}
        \label{Fig:SimulationOutputError}
    \end{subfigure}
\end{figure}
\begin{figure}[ht]\ContinuedFloat
    \centering
    \begin{subfigure}[b]{0.48\textwidth}
        \centering
        \includegraphics[width=8.8cm]{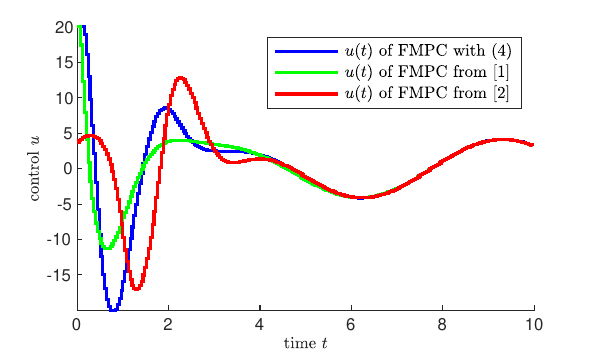}
        \caption{\vspace{-2mm}Control input}
        \label{Fig:SimulationControlInput}
    \end{subfigure}
    \caption{Simulation of system~\eqref{eq:ExampleMassOnCarSystem} under funnel MPC~\Cref{Algo:FMPC} and funnel MPC from~\cite{BergDenn22,BergDenn21}}
    \label{Fig:SimulationFunnelMPC}
    \vspace{-5mm}
\end{figure}
 
All simulations are performed with \textsc{Matlab} and the toolkit \textsc{CasADi} on 
the interval~$[0,10]$ and are depicted in~\Cref{Fig:SimulationFunnelMPC}.
The tracking errors {resulting from the application of the different funnel MPC schemes from~\cite{BergDenn22},~\cite{BergDenn21} and Algorithm~\ref{Algo:FMPC} to system~\eqref{eq:ExampleMassOnCarSystem}  are shown in~\Cref{Fig:SimulationOutputError}.
The corresponding control signals are displayed in~\Cref{Fig:SimulationControlInput}.
It is evident that all three control schemes achieve the control objective,
the evolution of the tracking error with in the performance boundaries given by~$\psi$.

Overall, the performance of all three funnel MPC schemes is comparable.
After $t=4$ the computed control signals and the corresponding tracking errors 
of all three control schemes are almost identical.
However, funnel MPC from~\cite{BergDenn22} requires feasibility constraints in the OCP to achieve initial and recursive feasibility; together with the more complex stage cost, this severely increases the computational effort. Furthermore, the parameters involved in the feasibility constraints are very hard to determine and usually (as in the simulations performed here) conservative estimates must be used. But then again, initial and recursive feasibility cannot be guaranteed. Concerning the funnel MPC scheme from~\cite{BergDenn21},
 it is still an open problem to show that it is initially and recursively feasible for systems with  relative degree larger than one.

 \section{Conclusion}

In the present paper we proposed a new model predictive control algorithm for a class of nonlinear systems with arbitrary relative degree, which achieves tracking of a reference signal with prescribed performance. The new funnel MPC scheme resolves the drawbacks of earlier approaches in~\cite{BergDenn21} (no proof of initial and recursive feasibility for relative degree larger than one) and~\cite{BergDenn22} (requirement of feasibility constraints, design parameters difficult to determine, high computational effort). All advantages of these approaches (no terminal costs or conditions, no requirements on the prediction horizon) are retained. Essentially, this solves the open problems formulated in the conclusions of~\cite{BergDenn22,BergDenn21}. Compared to previous works on funnel MPC, the class of nonlinear systems considered here includes systems with nonlinear delays and infinite-dimensional internal dynamics. An interesting question which remains for future research is, whether the weighted sum of the tracking error derivatives $e, \dot e, \ldots, e^{(r-1)}$ used in the cost functional in~\eqref{eq:FMPC_OCP} can be replaced by a sole error signal~$e$, when instead the prediction horizon~$T$ is chosen sufficiently long.
 
\bibliographystyle{elsarticle-harv}
\bibliography{references}
\end{document}